\documentclass{imanum}

\jno{drnxxx}
\received{20 April 2009}
\revised{23 March 2008}

\numberwithin{equation}{section}

\DeclareMathOperator{\Tr}{Tr}
\newcommand{\inner}[3][]{(#2,#3)_{#1}}
\newcommand{\norm}[2][]{\|{#2}\|_{{#1}}}
\newcommand{\snorm}[2][]{|{#2}|_{{#1}}}

\newcommand{\Norm}[2][]{\Big\|{#2}\Big\|_{{#1}}}
\newcommand{\abs}[1]{|#1|}

\newcommand{\dd}{ { \mathrm{d}} }
\newcommand{\ee}{ { \mathrm{e}} }

\newcommand{\pt}{\partial}

\newcommand{\cF}{{ \mathcal F}}
\newcommand{\cD}{{ \mathcal D}}

\newcommand{\cO}{{ \mathcal O}}

\newcommand{\IR}{{ \mathbf R}}

\newcommand{\IE}{{ \mathbf E}}

\newcommand{\HS}{\mathrm{HS}}

\begin{document}

\title{Finite element approximation of the linearized Cahn-Hilliard-Cook equation}
\shorttitle{Approximation of the linearized Cahn-Hilliard-Cook equation}

\author{%
  {\sc Stig Larsson}\thanks{Corresponding author. Supported by the
    Swedish Research Council (VR) and by the Swedish Foundation for
    Strategic Research (SSF) through GMMC, the Gothenburg Mathematical
    Modelling Centre.
    Email: stig@chalmers.se}\\[2pt]
  Department of Mathematical Sciences,
  Chalmers University of Technology \\
  and University of Gothenburg, SE--412 96 Gothenburg,
  Sweden\\[6pt]
  {\sc and}\\[6pt]
  {\sc Ali Mesforush}\thanks{Email:ali.mesforush@alumni.chalmers.se}\\[2pt]
  School of Mathematical Sciences, Shahrood University of Technology,
  Shahrood, Iran }

\shortauthorlist{S.~Larsson, A.~Mesforush}

\maketitle


\begin{abstract} {The linearized Cahn-Hilliard-Cook equation is
    discretized in the spatial variables by a standard finite element
    method. Strong convergence estimates are proved under suitable
    assumptions on the covariance operator of the Wiener process,
    which is driving the equation. The backward Euler time stepping is
    also studied. The analysis is set in a framework based on analytic
    semigroups.  The main effort is spent on proving detailed error
    bounds for the corresponding deterministic Cahn-Hilliard
    equation. The results should be interpreted as results on 
    approximation of the stochastic convolution, which is a part of
    the mild solution of the nonlinear Cahn-Hilliard-Cook equation.}
{Cahn-Hilliard-Cook equation, stochastic convolution, Wiener process,
  finite element method, backward Euler method, mean square, error
  estimate, strong convergence}
\end{abstract}


\section{Introduction} \label{sec:1}
When the Cahn-Hilliard equation (cf.~\citet{Cahn-Hill,Cahn-Hill1}) is
perturbed by noise, we obtain the so-called Cahn-Hilliard-Cook
equation (cf.~\citet{Dirk1,cook})
\begin{equation}\label{1.1}
\begin{aligned}
& \dd u-\Delta v\,\dd t=\dd W, && \textrm{for } x\in \cD,  \, t >0,\\
& v=-\Delta u + f(u), && \textrm{for }  x\in \cD, \, t>0,\\
& \frac{\partial u}{\partial n}=0,\ \frac{\partial\Delta u}{\partial
  n} = 0, 
&& \textrm{for } x \in \partial \cD, \, t>0,\\
& u(\cdot,0)=u_0,
\end{aligned}
\end{equation}
where $u = u(x,t)$, $\Delta =\sum_{i=1}^d \frac{\pt^2}{\pt x_i^2}$,
and $\frac{\pt}{\pt n}$ denotes the outward normal derivative on
$\pt\cD$.  We assume that $\cD$ is a bounded domain in $\IR^d$ for $d
\le 3$ with sufficiently smooth boundary.  A typical $f$ is $f(s) =
s^3 -s$. The purpose of this work is to study numerical approximation
by the finite element method of the linearized Cahn-Hilliard-Cook
equation, where $f=0$.

We use the semigroup framework of \citet{DaPratoZabczyk} in order to
give \eqref{1.1} a rigorous meaning.  Let $\norm{\cdot}$ and
$(\cdot,\cdot)$ denote the usual norm and inner product in the Hilbert
space $H = L_2(\cD)$ and let $H^s = H^s(\cD)$ be the usual Sobolev
space with norm $\norm[s]{\cdot}$. We also let $\dot{H}$ be the
subspace of $H$, which is orthogonal to the constants, that is, $\dot{H}
= \{v \in H:(v,1) =0 \}$, and we let $P\colon H\to\dot{H}$ be the
orthogonal projector. 

We define the linear operator $A = - \Delta$ with domain of definition
\begin{equation*}
 D(A) = \Big\{ v \in H^2 : \frac{ \pt v}{\pt n} =0 \  \text{on } \pt \cD \Big\}.
\end{equation*}
Then $A$ is a selfadjoint, positive definite, unbounded linear operator
on $\dot{H}$ with compact inverse. When it is considered as an
unbounded operator on $H$, it is positive semidefinite with an
orthonormal eigenbasis $\{\varphi_j\}_{j=0}^\infty$ and corresponding
eigenvalues $\{\lambda_j\}_{j=0}^\infty$ such that
\begin{equation*}
0 = \lambda_0 < \lambda_1 \le \lambda_2 \le \cdots \le \lambda_j 
\le \cdots ,
\quad \lambda_j \to \infty.
\end{equation*}
The first eigenfunction is constant, $\varphi_0 =
\abs{\cD}^{-\frac{1}{2}}$.  We also define
\begin{align} \label{defnorm}
  \snorm[s]{v} = \norm{A^{\frac{s}{2}}Pv} 
  =\Big( \sum_{j=1}^\infty \lambda_j^s(v,\varphi_j)^2\Big)^{1/2}, 
\quad s\in\IR,  
\end{align}
and $\dot{H}^s = \{v\in\dot{H}:\snorm[s]{v}<\infty\}$ for $s\ge0$.
For $s<0$ we define $\dot{H}^s$ to be the closure of $\dot{H}$ with
respect to $\snorm[s]{\cdot}$.  Then $\dot{H}^0=\dot{H}$ and
$\norm{v}^2=\snorm[0]{v}^2+(v,\varphi_0)^2$.  It is well known that,
for integer $s \ge 0$, $\dot{H}^s$ is a subspace of $H^s \cap \dot H$
characterized by certain boundary conditions and that the norms
$\abs{\cdot}_s$ and $\norm[s]{\cdot}$ are equivalent on
$\dot{H}^s$. In particular, we have $\dot{H}^1 = H^1 \cap \dot H$ and
the norm $\snorm[1]{v} = \norm{A^{\frac{1}{2}} v} = \norm{ \nabla v}$
is equivalent to $\norm[1]{v}$ on $\dot{H}^1$.

For $v \in H$ we define 
\begin{align*}
\ee^{-tA^2} v 
= \sum_{j=0}^{\infty} \ee^{-t\lambda_j^2} (v,\varphi_j)\varphi_j.
\end{align*}
Then $\{ E(t)\}_{t \ge 0} = \{ \ee ^{-tA^2}\}_{t \ge 0}$  is the
analytic semigroup on $H$ generated by $-A^2$.  We note that 
\begin{align*}
E(t) v 
= \sum_{j=1}^{\infty} \ee^{-t\lambda_j^2} (v,\varphi_j)\varphi_j
+(v,\varphi_0)\varphi_0
= E(t)Pv+(I-P)v,  
\end{align*}
where $(I-P)v=|\cD|^{-1}\int_{\cD}v\,\dd x$ is the average of $v$. 

Let $(\Omega,\cF,\mathbf{P},\lbrace\cF_t\rbrace_{t \ge 0})$ be a
filtered probability space, let $Q$ be a selfadjoint, positive
semidefinite, bounded linear operator on $H$, and let $\lbrace
W(t)\rbrace_{t \ge 0}$ be an $H$-valued $Q$-Wiener process adapted to
the filtration $\lbrace\cF_t\rbrace_{t \ge 0}$.

Now the Cahn-Hilliard-Cook equation \eqref{1.1} may be written formally
\begin{equation}\label{1.3}
\dd X(t) + A^2 X(t)\, \dd t + Af(X(t)) \, \dd t = \dd W(t), \quad t > 0;
\quad X(0) = X_0.
\end{equation}
The semigroup framework of \citet{DaPratoZabczyk} gives a rigorous
meaning  to this in terms of the mild solution, which satisfies the integral equation
\begin{equation*}
 X(t) = E(t) X_0 - \int_0^t E(t-s) Af(X(s)) \, \dd s + \int_0^t E(t-s) \, \dd W(s),
\end{equation*}
where $\int_0^t\cdots\, \dd W(s)$ denotes the $H$-valued It\^o
integral. Existence and uniqueness of solutions is proved in
\citet{Prato}.  This is based on the natural splitting of the
solution as $X(t)=Y(t)+W_A(t)$, where
\begin{align}  \label{stochconvol}
  W_A(t)=\int_0^t E(t-s) \, \dd W(s)
\end{align}
is a stochastic convolution, and where 
\begin{align*}
  Y(t)=E(t) X_0 - \int_0^t E(t-s) Af(X(s)) \, \dd s 
\end{align*}
satisfies the random evolution problem
\begin{align*} 
 \dot{Y}(t) + A^2 Y(t) + A f\big(Y(t)+W_A(t)\big) = 0,\quad t >0; \quad
 Y(0)=X_0.
\end{align*}
The study of the stochastic convolution $W_A(t)$ is thus a first step
towards the study of the nonlinear problem.

In this work we therefore study numerical approximation of the linearized
Cahn-Hilliard-Cook equation
\begin{equation}\label{1.4}
  \dd X + A^2 X \, \dd t = \dd W,\quad t >0; \quad
  X(0) = X_0,
\end{equation}
with the mild solution
\begin{equation}\label{mildX}
X(t) = E(t) X_0 + \int_0^t E(t-s) \, \dd W(s) .
\end{equation} 
The nonlinear equation is studied in a forthcoming paper \citet{KLMchc}.
We remark that a linearized equation of the form \eqref{1.4}, but with
$A^2$ replaced by $A^2+A$ is studied by numerical simulation in the
physics literature \citet{elder,klein}.

For the approximation of the Cahn-Hilliard equation we follow the
framework of \citet{stig}. We assume that we have a family
$\{S_h\}_{h>0}$ of finite-dimensional approximating subspaces of
$H^1$. Let $P_h \colon H \to {S}_h$ denote the orthogonal projector.
We then define $\dot{S}_h = \{ \chi \in S_h : (\chi , 1) = 0\}$.  The
operator $A_h \colon {S}_h \to \dot{S}_h$ (the ``discrete Laplacian'')
is defined by
\begin{equation*}
 (A_h \chi , \eta) = (\nabla \chi , \nabla \eta), \quad \forall
 \chi\in S_h,\, \eta \in \dot{S}_h,
\end{equation*}
The operator $A_h$ is selfadjoint, positive definite on $\dot{S}_h$,
positive semidefinite on $S_h$, and $A_h$ has an orthonormal
eigenbasis $\lbrace \varphi_{h,j} \rbrace_{j=0}^{N_h}$ with
corresponding eigenvalues $\lbrace \lambda_{h,j} \rbrace_{j=0}^{N_h}$.
We have
\begin{equation*}
0 = \lambda_{h,0} < \lambda_{h,1} \le \cdots \le \lambda_{h,j} \le
\cdots \le \lambda_{h,N_h},
\end{equation*}
and $\varphi_{h,0} = \varphi_0 = \abs{\cD}^{-\frac{1}{2}}$. 
Moreover, we define $E_h(t) \colon S_h \to S_h$ by
  \begin{equation*}
\begin{split}
  E_h(t)v_h=\ee^{-t A_h^2} v_h
  = \sum_{j=1}^{N_h} \ee^{-t\lambda_{h,j}} 
   \inner{v_h}{\varphi_{h,j}}\varphi_{h,j} 
   + \inner{v_h}{\varphi_0}\varphi_{0}. 
  \end{split}
\end{equation*}
Then $\{ E_h(t)\}_{t \ge 0}$ is the analytic semigroup generated by
$-A_h^2$.  Clearly, $P_h \colon \dot{H} \to \dot{S}_h$ and
\begin{equation*}
E_h(t)P_h v = E_h(t)P_h P v + (I-P)v. 
\end{equation*}

The finite element approximation of the linearized Cahn-Hilliard-Cook
equation \eqref{1.4} is: Find $X_h(t) \in {S}_h$ such that,
\begin{equation}\label{LinCHC}
 \dd X_h + A_h^2 X_h \, \dd t = P_h \, \dd  W, \quad t >0; \quad
 X_h(0) = P_h X_0.
\end{equation}
The mild solution of \eqref{LinCHC} is
\begin{equation}\label{mildXh}
 X_h(t) = E_h(t) P_h X_0 + \int_0^t E_h(t-s) P_h \,\dd W(s).
\end{equation}
We note that 
\begin{align*}
\int_0^t E(t-s)(I-P) \, \dd W(s) =(I-P) \int_0^t \, \dd W(s) = (I-P)W(t) ,  
\end{align*}
so that \eqref{mildX} can be written
\begin{align} \label{fix}
X(t) = E(t)P X_0 +(I-P)X_0 + \int_0^t E(t-s)P \, \dd W(s)
 +(I-P)W(t), 
\end{align} 
and similarly, for \eqref{mildXh},
\begin{align*}
X_h(t)
= E_h(t)P_hP X_0 +(I-P)X_0 
+ \int_0^t E_h(t-s)P_hP \, \dd W(s)
 +(I-P)W(t).   
\end{align*}
Therefore, the error analysis can be based on the formula 
\begin{align} \label{errorX}
X_h(t)-X(t)= \big(E_h(t)P_h-E(t)\big)P X_0
+\int_0^t \big(E_h(t-s)P_h-E(t-s)\big)P \, \dd W(s),  
\end{align}
and it is sufficient to work in the spaces $\dot{H}$ and $\dot{S}_h$.
Note that the numerical computations are carried out in $S_h$ and that
$\dot{S}_h$ is only used in the analysis.

Let $k = \delta t$ be a timestep, $t_n = nk,\, \delta X_{h,n} =
X_{h,n} - X_{h,n-1},\, \delta W_n = W(t_n) - W(t_{n-1})$, and apply
Euler's method to \eqref{LinCHC} to get
\begin{equation}\label{DeltaXhn}
 \delta X_{h,n} + A_h^2 X_{h,n}\,\delta t = P_h \,\delta W_n, \quad
 n\ge1;\quad X_{h,0}=P_hX_0.
\end{equation}
With $E_{kh} = (I + kA_h^2)^{-1}$ we obtain a discrete variant of the mild solution
\begin{equation*}
  X_{h,n} = E_{kh}^n P_h X_0 + \sum_{j=1}^n E_{kh}^{n-j+1} P_h \,\delta W_j.
\end{equation*}

In Section \ref{sec:2} we assume that $\lbrace S_h \rbrace_{h > 0}$
admits an error estimate of order $\cO(h^r)$ as the mesh parameter $h
\to 0$ for some integer $r \ge 2$. Then we show error estimates for
the semigroup $E_h(t)$ with minimal regularity requirement. More
precisely, in Theorem \ref{lem1} we show, for $\beta \in [1,r]$ and
all $t\ge0$,
\begin{align*}
 & \norm{F_h(t)v} \le Ch^{\beta}\,\snorm[\beta]{v}, \quad  v \in \dot{H}^{\beta},\\
 & \Big( \int_0^t \norm{F_h(\tau) v}^2 \, \dd \tau\Big)^{\frac{1}{2}}  \le
    C\abs{\log h}h^{\beta} \,\snorm[\beta - 2]{v},
   \quad v\in \dot{H}^{\beta -2},
\end{align*}
where $F_h(t) = E_h(t) P_h - E(t)$ is the error operator in
\eqref{errorX}.  

Analogous estimates are obtained for the implicit Euler approximation in Theorem \ref{lemma2}.

In Section \ref{sec:3} we follow the technique developed in
\citet{yubin1,yubin2} and use these estimates to prove strong
convergence estimates for approximation of the linear
Cahn-Hilliard-Cook equation. Let $L_2(\Omega , \dot{H}^\beta)$ be the
space of square integrable $\dot{H}^\beta$-valued random variables
with norm
\begin{equation} \label{meansquare}
 \norm[L_2(\Omega,\dot{H}^\beta)]{X} 
= \Big( \IE\big\{ \snorm[\beta]{X}^2\big\}\Big)^{\frac{1}{2}}
= \Big( \int_{\Omega} \snorm[\beta]{X(\omega)}^2 \,\dd \mathbf{P}(\omega)\Big)^{\frac{1}{×2}},
\end{equation}
and let $\norm[\HS]{T}$ denote the Hilbert-Schmidt norm of linear
operators on $H$, $\norm[\HS]{T}^2 = \sum_{j=1}^{\infty}
\norm{T\phi_j}^2$, where $\lbrace \phi_j\rbrace_{j=1}^{\infty}$ is an
arbitrary orthonormal basis for $H$. In Theorem \ref{thm1} we study
the spatial regularity of the mild solution \eqref{mildX} and show
\begin{equation*}
 \norm[L_2(\Omega,\dot{H}^{\beta})]{X(t)} \le C\big(
 \norm[L_2(\Omega,\dot{H}^{\beta})]{X_0} + \norm[\HS]{A^{\frac{\beta
       -2}{2}}Q^{\frac{1}{2}}}\big)\quad \text{for } \beta \ge 0.
\end{equation*}
Moreover, in Theorem \ref{thm2} we show strong convergence for the mild solution $X_h$ in \eqref{mildXh}:
\begin{equation*}
 \|X_h(t)  - X(t)\|_{L_2(\Omega , H)}
   \le Ch^{\beta}\big( \norm[L_2(\Omega , \dot{H}^{\beta})]{X_0} 
+ \abs{\log h}\norm[\HS]{A^{\frac{\beta - 2}{2}}Q^{\frac{1}{2}}}
\big), 
\quad \beta \in [1,r].
\end{equation*}
In Theorem \ref{thm3} for the fully discrete case we obtain similarly, for $\beta \in [1,\min (r,4)]$,
\begin{align*}
  \|X_{h,n}   -  X(t_n)  \|_{L_2(\Omega , H)}
    \le \big(C\abs{\log h} h^{\beta} 
+ C_{k,\beta}k^{\frac{\beta}{4}}\big)\big( \norm[L_2(\Omega ,
\dot{H}^{\beta})]{X_0} 
+ \norm[\HS]{A^{\frac{\beta - 2}{2}}Q^{\frac{1}{2}}} \big),
\end{align*}
where $C_{\beta,k} = \frac{C}{4-\beta}$ for $\beta <4$ and
$C_{\beta,k} = C\abs{\log k}$ for $\beta =4$.

Note that these bounds are uniform with respect to $t\ge0$ and  $t_n\ge0$.

Our results require that
$\norm[\HS]{A^{\frac{\beta-2}{2}}Q^{\frac{1}{2}}} < \infty.$ In order
to see what this means we compute two special cases. For $Q=I$
(spatially uncorrelated noise, or space-time white noise), by using
the asymptotics $\lambda_j \sim j^{\frac{2}{d}}$, we have
\begin{equation*}
\norm[\HS]{A^{\frac{\beta-2}{2}}Q^{\frac{1}{2}}}^2 
= \norm[\HS]{A^{\frac{\beta-2}{2}}}^2 
= \sum_{j=1}^{\infty} \lambda_j^{\beta -2} 
\sim \sum_{j=1}^{\infty} j^{(\beta-2) \frac{2}{d}} < \infty,
\end{equation*}
if $\beta < 2 - \frac{d}{2}$. Hence, for example, $\beta <
\frac{1}{2}$ if $d = 3$. On the other hand, if $Q$ is of trace class,
$\Tr(Q) = \norm[\HS]{Q^{\frac{1}{2}}}^2 < \infty$, then we may take
$\beta = 2$.

There are few studies of numerical methods for the Cahn-Hilliard-Cook
equation. We are only aware of \citet{Weber}, in which convergence in
probability was proved for a difference scheme for the nonlinear
equation in multiple dimensions, and \citet{KossiorisZouraris}, where
strong convergence was proved for the finite element method for the
linear equation in 1-D.
\section{Error estimates for the Cahn-Hilliard semigroup}\label{sec:2}
We begin by recalling some inequalities. Let $\{
E(t)\}_{t \ge 0} = \{ \ee ^{-tA^2}\}_{t \ge 0}$ and $\{ E_h(t)\}_{t
  \ge 0} = \{ \ee ^{-tA_h^2}\}_{t \ge 0}$ be the semigroups generated
by $-A^2$ and $ -A_h^2$, respectively. By the smoothing property there
exist positive constants $c,C$, independent of $h$ and $t$, such that
\begin{align}
 & \norm{A_h^{2 \beta} E_h(t) P_h Pv} 
+ \norm{A^{2 \beta} E(t) P v} 
\le C t^{- \beta}\ee^{-ct} \norm{v},\quad \beta \ge 0, \label{bddness}\\
 & \int_0^t \norm{A_h E_h(s) P_h Pv}^2\, \dd s 
+ \int_0^t \norm{A E(s) P v}^2\, \dd s \le C \norm{v}^2\label{half}.
\end{align}
Let  $R_h \colon  \dot{H}^1 \to \dot{S}_h$ be the Ritz projector defined by
\begin{equation*}
(\nabla R_h v , \nabla \chi) = (\nabla v , \nabla \chi), \quad \forall \chi \in \dot{S}_h.
\end{equation*}
It is clear that $R_h = A_h^{-1} P_h A$. We assume that for some
integer $r \ge 2$, we have the error bound, with the norm defined in \eqref{defnorm},
\begin{equation}\label{ritzerr}
 \norm{R_h v - v } \le Ch^{\beta} \snorm[\beta]{v}, \quad  v \in \dot{H}^{\beta}, \, 1 \le \beta \le r.
\end{equation}
This holds with $r=2$ for the standard piecewise linear Lagrange
finite element method in a bounded convex polygonal domain $\cD$. For
higher order elements the situation is more complicated and we refer
to standard texts on the finite element method.  In the next theorem
we prove error estimates for the Cahn-Hilliard semigroup in the
semidiscrete case.
\begin{theorem}\label{lem1}
 Set $ F_h(t) = E_h(t)P_h - E(t)$. Then there are $h_0$ and $C$, such
 that for $h\le h_0$, $1 \le \beta \le r$ and $t \ge 0$, we have
 \begin{align}
 & \norm{F_h(t)v} \le Ch^{\beta}\,\snorm[\beta]{v}, 
\quad  v \in \dot{H}^{\beta} \label{first},\\
 & \Big( \int_0^t \norm{F_h(\tau) v}^2 \, \dd \tau\Big)^{\frac{1}{2}}  
\le  C\abs{\log h}h^{\beta}\, \snorm[\beta - 2]{v},
   \quad v\in \dot{H}^{\beta -2}\label{second}.
\end{align}
\end{theorem}
Note that $F_h(t)v=F_h(t)Pv$ for $v\in H$, so that it is sufficient to
take $v\in\dot H$.  The reason why we assume $\beta \ge 1$ is that in
\eqref{second} we need at least $v \in \dot{H}^{-1}$ for $E_h(t)P_h v$
to be defined.  The bound \eqref{first} can be found in \citet{stig},
but \eqref{second} is new.  We adapt the technique from \citet{stig} to
provide a self-contained proof.  

\begin{proof}
Let $u(t) = E(t) v$, $u_h(t) = E_h(t) P_h v$ be the solutions of 
\begin{align}\label{1.2}
u_t + A^2 u &= 0, \quad t>0; \quad 
u(0) = v,  \\
\label{2.3}
  u_{h,t} + A_h^2 u_h  &= 0, \quad t >0; \quad
   u_h(0)= P_h v.
\end{align}
Here $u_t$ denotes the time derivative. Set $e(t)=u_h(t) - u(t)$. We
want to prove that
\begin{equation*}
\begin{aligned}
& \norm{e(t)} \le Ch^{\beta}\, \snorm[\beta]{v}, \quad v \in \dot{H}^{\beta},\\
& \Big( \int_0^t \norm{e(\tau)}^2 \, \dd \tau \Big)^{\frac{1}{2}}
   \le C\abs{\log h}h^{\beta}\,\snorm[\beta - 2]{v}, \quad v \in \dot{H}^{\beta - 2}.
\end{aligned}
\end{equation*}
Let $G=A^{-1}P$ and $G_h = A_h^{-1} P_hP$. Apply $G$ to \eqref{1.2} to
get $G u_t + Au = 0$, and apply $G_h^2$ to \eqref{2.3} to get $G_h^2
u_{h,t} + u_h =0$. Hence,
\begin{equation*}
G_h^2 e_t + e  =-G_h^2 u_t - u  +G_h(Gu_t+Au) = (G_h A - I) u - G_h (G_h A -  I )G u_t,
\end{equation*}
that is,
\begin{equation}\label{main}
G_h^2 e_t +  e =  \rho +G_h \eta,
\end{equation}
where $\rho = (R_h-I)u ,\eta =-(R_h - I)G u_t$. Take the inner product of \eqref{main} by $e_t$ to get
\begin{equation*}
 \norm{G_h e_t}^2 + \frac{1}{2}  \frac{\dd}{\dd t}\norm{e}^2 =  (\rho , e_t) +(\eta , G_h e_t),
\end{equation*}
Since
$(\eta , G_h e_t) \le \norm{\eta} \norm{ G_h e_t} \le \frac{1}{2} \norm{\eta}^2 + \frac{1}{2} \norm{G_h e_t}^2$,
we obtain
\begin{equation*}
\norm{G_h e_t}^2 + \frac{\dd}{\dd t}\norm{e}^2 \le 2  (\rho,e_t) +  \norm{\eta}^2.
\end{equation*}
Multiply this inequality by $t$ to get 
\begin{align*}
t\norm{G_h e_t}^2 + t \frac{\dd}{\dd t}\norm{e}^2 
\le 2  t (\rho,e_t) +  t \norm{\eta}^2.
\end{align*}
For compactness of notation we write $t\norm{\eta}^2$ instead of
$t\norm{\eta(t)}^2$ and similarly for the other terms. 
Note that
\begin{align*}
  t \frac{\dd}{\dd t} \norm{e}^2= \frac{\dd}{\dd t}\big(t \norm{e}^2\big) - \norm{e}^2,
  \quad 
  t(\rho , e_t)=\frac{\dd}{\dd t}\big(t(\rho,e)\big)-(\rho , e)-t(\rho_t,e),
\end{align*}
so that
\begin{equation*}
t \norm{G_h e_t}^2 + \frac{\dd}{\dd t}\big(t \norm{e}^2\big) \le 2 \frac{\dd}{\dd t}\big(t(\rho,e)\big)+2\abs{(\rho , e)}+2 \abs{t(\rho_t,e)}+t\norm{\eta}^2+ \norm{e}^2.
\end{equation*}
But
\begin{align*}
\abs{(\rho ,e)} &\le \norm{\rho} \norm{e} \le \frac{1}{2}\norm{\rho}^2+\frac{1}{2}\norm{e}^2,\\
\abs{t(\rho_t,e)}&\le t \norm{\rho_t}\norm{e} \le \frac{1}{2} t^2 \norm{\rho_t}^2+\frac{1}{2}\norm{e}^2.
\end{align*}
Hence, 
\begin{equation*}
 t \norm{G_h e_t}^2 +  \frac{\dd}{\dd t}\big( t \norm{e}^2\big)
\le 2 \frac{\dd}{\dd t}\big(t(\rho,e)\big) +  \norm{\rho}^2 + t^2 \norm{\rho_t}^2 + t \norm{\eta}^2+ 3\norm{e}^2.
\end{equation*}
Integrate over $[0,t]$ and use Young's inequality to get
\begin{align*}
\int_0^t \tau \norm{G_h e_t}^2 \, \dd \tau +  t \norm{e}^2
& \le  2 t \norm{\rho}^2 + \frac{1} {2} t\norm{e}^2 + \int_0^t \norm{\rho}^2 \, \dd \tau +  \int_0^t \tau^2 \norm{\rho_t}^2 \, \dd \tau \\
&\quad + \int_0^t \tau \norm{\eta}^2 \, \dd \tau + 3  \int_0^t \norm{e}^2 \, \dd \tau.
\end{align*}
Hence,
\begin{equation}\label{3.4}
\begin{aligned}
 t \norm{e}^2
 & \le Ct\norm{\rho}^2 +C \int_0^t \big( \norm{\rho}^2 + \tau^2 \norm{\rho_t}^2 +  \tau \norm{\eta}^2 + \norm{e}^2 \big)\,\dd \tau.
\end{aligned}
\end{equation}
We must bound $\int_0^t \norm{e}^2\,\dd \tau$. Multiply \eqref{main} by $e$ to get
\begin{equation*}
  \frac{1}{2}\frac{\dd}{\dd t}\norm{G_he}^2 +  \norm{e}^2
  \le  \norm{\rho}\norm{e}+\norm{\eta}\norm{G_h e}
  \le \frac{1}{2}\norm{\rho}^2+\frac{1}{2}\norm{e}^2+\norm{\eta}\max_{0 \le \tau \le t}\norm{G_h e},
\end{equation*}
so that
\begin{equation}\label{3.5}
\frac{\dd}{\dd t}\norm{G_he}^2 +  \norm{e}^2
\le  \norm{\rho}^2+2\norm{\eta}\max_{0 \le \tau \le t}\norm{G_h e}.
\end{equation}
Integrate \eqref{3.5}, note that $G_he(0) = A_h^{-1} P_h(P_h - I) v =0$, to get
\begin{equation*}
  \norm{G_h e}^2 + \int_0^t \norm{e}^2 \, \dd \tau
  \le  \int_0^t \norm{\rho}^2 \, \dd \tau + \max_{0 \le \tau \le t}\norm{G_h e}^2+
   \Big(\int_0^t \norm{\eta} \, \dd \tau\Big)^2.
\end{equation*}
Hence, since $t$ is arbitrary,
\begin{equation}\label{3.6}
  \int_0^t \norm{e}^2 \, \dd \tau \le
\int_0^t \norm{\rho}^2 \, \dd \tau + \Big(\int_0^t \norm{\eta} \, \dd \tau\Big)^2.
\end{equation}
We insert \eqref{3.6} in \eqref{3.4} and conclude
\begin{equation}\label{3.7}
 t \norm{e}^2
 \le Ct\norm{\rho}^2+C\int_0^t \big( \norm{\rho}^2+ \tau^2 \norm{\rho_t}^2 + \tau\norm{\eta}^2\big)\,\dd\tau 
\quad + C\Big( \int_0^t \norm{\eta} \, \dd \tau\Big)^2.
\end{equation}
We compute the terms in the right hand side. With $v \in \dot{H}^{\beta}$, recalling $\rho = (R_h - I)u$ and using  \eqref{ritzerr}, we have
\begin{equation}\label{ett}
\begin{aligned}
\norm{\rho(t)}
\le Ch^{\beta} \snorm[\beta]{u(t)}
\le Ch^{\beta} \norm{E(t) A^{\frac{\beta}{2}}v}
\le Ch^{\beta} \norm{A^{\frac{\beta}{2}}v}
\le Ch^{\beta}\snorm[\beta]{v},
\end{aligned}
\end{equation}
so that,
\begin{equation*}
t \norm{\rho}^2 \le Ch^{2\beta} t \snorm[\beta]{v}^2,\quad
\int_0^t \norm{\rho}^2\,\dd \tau \le Ch^{2\beta}t \snorm[\beta]{v}^2.
\end{equation*}
Similarly, by \eqref{bddness},
\begin{equation*}
\norm{\rho_t(t)}
 \le Ch^{\beta} \snorm[\beta]{u_t(t)}
\le  C h^{\beta} \norm{A^2 E(t)A^{\frac{\beta}{2}}v}
 \le Ch^{\beta} t^{-1} \snorm[\beta]{v},
\end{equation*}
so that
\begin{equation}\label{num2}
 \int_0^t \tau^2 \norm{\rho_t}^2\,\dd \tau \le Ch^{2\beta} t \snorm[\beta]{v}^2.
\end{equation}
Moreover, since $\eta = -(R_h - I)Gu_t=(R_h-I)GA^2E(t)v$,
\begin{equation*}
 \norm{\eta(t)}
  \le Ch^{\beta} \snorm[\beta]{Gu_t(t)}
  \le  C  h^{\beta} \norm{A E(t) A^{\frac{\beta}{2}} v}
  \le Ch^{\beta} t^{-\frac{1}{2}} \snorm[\beta]{v},
\end{equation*}
so that
\begin{align*}
  \Big(\int_0^t \norm{\eta} \, \dd \tau \Big)^2 \le Ch^{2\beta} t \snorm[\beta]{v}^2,\quad
  \int_0^t \tau \norm{\eta}^2 \, \dd \tau \le Ch^{2\beta} t \snorm[\beta]{v}^2.
\end{align*}
By inserting these in \eqref{3.7} we conclude
\begin{equation*}
 t\norm{e}^2 \le Ch^{2\beta}t \snorm[\beta]{v}^2,
\end{equation*}
which proves \eqref{first}.

To prove \eqref{second} we recall \eqref{3.6} and let $v \in \dot{H}^{\beta-2}$. By using \eqref{ritzerr} and \eqref{half} we obtain
\begin{equation}\label{2.18}
\begin{split}
\int_0^t \norm{\rho}^2 \, \dd \tau
 & \le Ch^{2\beta} \int_0^t \snorm[\beta]{u}^2\,\dd \tau
 = Ch^{2\beta} \int_0^t \norm{AE(\tau) A^{\frac{\beta-2}{2}}v}^2\,\dd \tau\\
 & \le Ch^{2\beta} \snorm[\beta-2]{v}^2.
 \end{split}
\end{equation}
Now we compute $ \int_0^t \norm{\eta} \, \dd \tau$. To this end we assume first $1 < \beta \le r$ and let  $1 \le \gamma < \beta$. By using \eqref{bddness} and \eqref{ritzerr} we get
\begin{align*}
\int_0^t \norm{\eta}\,\dd \tau
& \le Ch^{\gamma} \int_0^t \snorm[\gamma]{Gu_t}\,\dd \tau
 = Ch^{\gamma}\int_0^t\norm{A^{2-\frac{\beta-\gamma}{2}}E(\tau) A^{\frac{\beta-2}{2}}v}\,\dd \tau\\
& \le Ch^{\gamma}\int_0^t \tau^{-1+\frac{\beta-\gamma}{4}}\ee^{-c\tau}\,\dd \tau \,\snorm[\beta-2]{v},\\
\end{align*}
where, since $0 < \beta - \gamma \le r-1$,
\begin{equation*}
\int_0^t \tau^{-1+\frac{\beta-\gamma}{4}}\ee^{-c\tau}\,\dd \tau
= \frac{4}{\beta - \gamma} \int_0^{t^{\frac{\beta-\gamma}{4}}} \ee^{-cs^{\frac{4}{\beta - \gamma}}}\,\dd s
\le \frac{C}{\beta-\gamma} \int_0^{\infty} \ee^{-cs^{\frac{4}{r-1}}}\,\dd s.
\end{equation*}
Hence, with $C$ independent of $\beta$,
\begin{equation}\label{star}
\int_0^t \norm{\eta}\,\dd \tau \le \frac{Ch^{\gamma}}{\beta - \gamma} \snorm[\beta-2]{v}.
\end{equation}
Now let $\frac{1}{\beta-\gamma}=\abs{\log h}=-\log h$, so $\gamma \to \beta$ as $h \to 0$, and
\begin{equation*}
\gamma \log h = (\gamma -\beta+\beta)\log h = 1 + \beta \log h.
\end{equation*}
Therefore we have
\begin{equation*}
\frac{h^{\gamma}}{\beta - \gamma}
= \abs{\log h}\ee ^{\gamma\log h}\\
= \abs{\log h} \ee^{1+\beta\log h}\\
\le C\abs{\log h}h^{\beta}.
\end{equation*}
Put this in \eqref{star} to get, for $1 < \beta \le r$,
\begin{equation}\label{yek}
\int_0^t \norm{\eta}\,\dd \tau \le  Ch^{\beta}\abs{\log h}\snorm[\beta-2]{v},
\end{equation}
and hence also for $1 \le \beta \le r$, because $C$ is independent of $\beta$.
Finally, we put \eqref{2.18} and \eqref{yek} in \eqref{3.6} to get
\begin{equation*}
 \Big(\int_0^t \norm{e}^2 \,\dd\tau\Big)^{\frac{1}{2}}
\le
 C\abs{\log h}h^{\beta}\snorm[\beta-2]{v},
\end{equation*}
which is \eqref{second}. The proof is complete. 
\end{proof}

Now we turn to the fully discrete case. The backward Euler method applied to
\begin{align*}
 u_{h,t}+  A_h^2u_h =  0, \quad t >0;\quad
 u_h(0) = P_h v,
\end{align*}
defines $U_n \in {S}_h$ by
\begin{equation}\label{3.15}
 \pt U_n+  A_h^2 U_n = 0, \quad n \ge 1;\quad
 U_0 = P_h v,
\end{equation}
where $\pt U_n = \frac{1}{k} (U_n - U_{n-1})$. Denoting $E_{kh}^n = (I
+ k A_h^2)^{-n}$, we have $U_n = E_{kh}^n v$.  The next theorem
provides error estimates for the Euler approximation of the
Cahn-Hilliard semigroup.
\begin{theorem}\label{lemma2}
  Set $F_n = E_{kh}^n P_h - E(t_n)$. Then there are $h_0, k_0$ and $C$, such
 that for $h\le h_0$,  $k\le k_0$, $1 \le \beta \le \min(r,4)$, and $n \ge 1$, we have
 \begin{align}
  & \norm{F_n v}
  \le
  C(h^{\beta} + k^{\frac{\beta}{4}})\snorm[\beta]{v}, \quad v \in \dot{H}^{\beta},\label{lemp1}\\
  & \Big( k \sum_{j=1}^n \norm{F_j v}^2 \Big)^{\frac{1}{2}}
  \le
  \big( C\abs{\log h}h^{\beta} + C_{\beta,k}k^{\frac{\beta}{4}}\big)\snorm[\beta-2]{v},\quad v \in \dot{H}^{\beta-2},\label{lemp2}
 \end{align}
where $C_{\beta,k} = \frac{C}{4-\beta}$ for $\beta <4$ and $C_{\beta,k} = C\abs{\log k}$ for  $\beta =4$.
\end{theorem}
\begin{proof}
Let $G$ and $G_h$ be as in the proof of Theorem \ref{lem1}. With $e_n = U_n -u_n = E_{kh}^n P_h v - E(t_n) v$, we get
\begin{equation}\label{3.26}
G_h^2 \pt e_n +  e_n =  \rho_n + G_h \eta_n + G_h \delta_n,
\end{equation}
where $u_n = u(t_n)$, $u_{t,n} = u_t(t_n)$ and
\begin{equation*}
\rho_n = (R_h - I)u_n,\quad \eta_n = -(R_h - I)G \pt u_n,\quad \delta_n = -G(\pt u_n - u_{t,n}).
\end{equation*}
Multiply \eqref{3.26} by $\pt e_n$ and note that
\begin{equation*}
(\eta_n,G_h\pt e_n) \le \norm{\eta_n}^2 + \frac{1}{4} \norm{G_h \pt e_n}^2, \,
(\delta_n,G_h\pt e_n) \le \norm{\delta_n}^2 + \frac{1}{4} \norm{G_h \pt e_n}^2,
\end{equation*}
to get
\begin{equation}\label{3.27}
\norm{G_h \pt e_n}^2 + 2  (e_n,\pt e_n) \le  2  (\rho_n,\pt e_n)+ 2 \norm{\eta_n}^2+ 2 \norm{\delta_n}^2.
\end{equation}
We have the following identities
\begin{align}
 \label{3.28}
\pt(a_n b_n)
& = (\pt a_n)b_n + a_{n-1}(\pt b_n)\\
\label{3.29}
& = (\pt a_n)b_n+a_n(\pt b_n) - k (\pt a_n)(\pt b_n).
\end{align}
By using \eqref{3.29} we have
\begin{align*}
& 2  (e_n , \pt e_n)=  \pt \norm{e_n}^2 + k  \norm{\pt e_n}^2,\\
&  (\rho_n , \pt e_n) =   \pt (\rho_n , e_n) - (\pt \rho_n , e_n) +  k (\pt \rho_n , \pt e_n).
\end{align*}
Put these in \eqref{3.27} and  cancel $k \norm{\pt e_n}^2$ to get
\begin{equation*}
\norm{G_h \pt e_n}^2 +  \pt \norm{e_n}^2 \le 2  \pt (\rho_n,e_n) - 2 (\pt \rho_n , e_n) + k  \norm{\pt \rho_n}^2 + 2 \norm{\eta_n}^2 + 2\norm{\delta_n}^2.
\end{equation*}
Multiply this by $t_{n-1}$, and note that $k \le t_{n-1}$ for $n \ge
2$, so that for $n \ge 1$ we have 
\begin{equation}\label{3.33}
\begin{split}
t_{n-1}\norm{G_h \pt e_n}^2 +  t_{n-1} \pt \norm{e_n}^2
 &\le  2t_{n-1} \pt (\rho_n,e_n)
- 2  t_{n-1}(\pt \rho_n , e_n) +  t_{n-1}^2 \norm{\pt \rho_n}^2\\
& \quad + 2 t_{n-1}\norm{\eta_n}^2 + 2t_{n-1}\norm{\delta_n}^2.
\end{split}
\end{equation}
By \eqref{3.28} we have
\begin{align*}
 t_{n-1} \pt \norm{e_n}^2 &=  \pt(t_n \norm{e_n}^2) -  \norm{e_n}^2, \\
 2 t_{n-1}\pt(\rho_n,e_n) &= 2 \pt(t_n(\rho_n,e_n)) - 2 (\rho_n,e_n). 
\end{align*}
Put these in \eqref{3.33} to get
\begin{equation}\label{3.36}
\begin{split}
t_{n-1}  \norm{G_h \pt e_n}^2 +  \pt (t_n \norm{e_n}^2)
& \le  C \big(\pt (t_n(\rho_n,e_n))+\norm{\rho_n}^2 
+ t_{n-1}^2 \norm{\pt \rho_n}^2 + \norm{e_n}^2\big)\\
& \quad + C\big( t_{n-1}\norm{\eta_n}^2 + t_{n-1}\norm{\delta_n}^2 \big).
\end{split}
\end{equation}
Note that
\begin{align}
  k \sum_{j=1}^n \pt\big(t_j \norm{e_j}^2\big) = t_n\norm{e_n}^2,\label{3.37} \quad
  k \sum_{j=1}^n \pt\big(t_j(\rho_j,e_j)\big) = t_n(\rho_n,e_n).
\end{align}
By summation in \eqref{3.36} and using \eqref{3.37} we get
\begin{equation}\label{3.39}
\begin{split}
k \sum_{j=1}^n t_{j-1} \norm{G_h \pt e_j}^2 +  t_n \norm{e_n}^2
&\le  C  t_n \norm{\rho_n}^2
+ C  k \sum_{j=1}^n \big( \norm{\rho_j}^2 
+ t_{j-1}^2 \norm{\pt \rho_j}^2 + \norm{e_j}^2\big)\\
& \quad+ Ck  \sum_{j=1}^n \big( t_{j-1} \norm{\eta_j}^2 + t_{j-1} \norm{\delta_j}^2 \big).
\end{split}
\end{equation}
Now we estimate $k\sum_{j=1}^n \norm{e_j}^2$. Multiply \eqref{3.26} by $e_n$ to get
\begin{equation}\label{computeej}
2(G_h^2 \pt e_n , e_n) + \norm{e_n}^2 \le \norm{\rho_n}^2 + 2\big( \norm{\eta_n} + \norm{\delta_n} \big) \norm{G_h e_n}.
\end{equation}
By \eqref{3.29} we have
\begin{equation}
2(G_h^2 \pt e_n , e_n) = 2 (\pt G_h e_n , G_h e_n) = \pt \norm{G_h e_n}^2 + k \norm{\pt G_h e_n}^2.
\end{equation}
By summation in \eqref{computeej} and using $G_h e_0 = 0$, we get
\begin{equation*}
  \norm{G_h e_n}^2 +  k \sum_{j=1}^n \norm{e_j}^2 
    \le  k \sum_{j=1}^n \norm{\rho_j}^2 
        + \frac{1}{2} \max_{j \le n} \norm{G_h e_j}^2
      + 2 \Big( k \sum_{j=1}^n \big ( \norm{\eta_j} 
     + \norm{\delta_j} \big) \Big)^2.
\end{equation*}
Hence,
\begin{equation}\label{3.41}
 k \sum_{j=1}^n \norm{e_j}^2 \le  k \sum_{j=1}^n \norm{\rho_j}^2 + 2 \Big( k \sum_{j=1}^n \big ( \norm{\eta_j}+ \norm{\delta_j} \big) \Big)^2.
\end{equation}
By putting \eqref{3.41} in \eqref{3.39} we get
\begin{equation}\label{3.42}
\begin{split}
 t_n \norm{e_n}^2 &\le
  \,C  t_n \norm{\rho_n}^2
 + C k \sum_{j=1}^n \Big( \norm{\rho_j}^2 + t_{j-1}^2  \norm{\pt \rho_j}^2 +  t_{j-1}\norm{\eta_j}^2 + t_{j-1}\norm{\delta_j}^2  \Big)\\
 &\quad + C \Big( k \sum_{j=1}^n \big ( \norm{\eta_j}+ \norm{\delta_j} \big) \Big)^2.
\end{split}
\end{equation}
We compute these terms. With $v \in \dot{H}^{\beta}$ we have by \eqref{ett},
\begin{equation}\label{set1}
\norm{\rho_n}^2 \le Ch^{2\beta} \snorm[\beta]{v}^2,\quad
k\sum_{j=1}^n \norm{\rho_j}^2 \le Ch^{2\beta} t_n \snorm[\beta]{v}^2.
\end{equation}
By using the Cauchy-Schwartz  inequality we have
\begin{align*}
k\sum_{j=1}^n t_{j-1}^2\norm{\pt \rho_j}^2
 &= k\sum_{j=2}^n t_{j-1}^2 \Norm{\frac{1}{k} \int_{t_{j-1}}^{t_j} \rho_t \,\dd \tau}^2\\
& \le \sum_{j=2}^n\Big( t_{j-1}^2 \frac{1}{k} \int_{t_{j-1}}^{t_j} \tau^{-2}\dd \tau \int_{t_{j-1}}^{t_j} \tau^{2} \norm{\rho_t(\tau)}^2\, \dd \tau \Big)
 \le \int_0^{t_n} \tau^2 \norm{\rho_t}^2 \dd \tau.
\end{align*}
Hence, by \eqref{num2},
\begin{equation}\label{num3}
k\sum_{j=1}^n t_{j-1}^2\norm{\pt \rho_j}^2 \le Ch^{2\beta} t_n \snorm[\beta]{v}^2.
\end{equation}
By using \eqref{ritzerr} and \eqref{bddness} we have
\begin{equation*}
\begin{split}
\norm{\eta_j} & \le Ch^{\beta} \snorm[\beta]{G \pt u_j}
 \le \frac{Ch^{\beta}}{k} \Norm{ \int_{t_{j-1}}^{t_j} A E(\tau) A^{\frac{\beta}{2}} v \,\dd \tau}\\
 &\le \frac{Ch^{\beta}}{k} \int_{t_{j-1}}^{t_j} \tau^{-\frac{1}{2}}\,\dd \tau \norm{A^{\frac{\beta}{2}}v}
 \le \frac{Ch^{\beta}}{k} (\sqrt{t_j} - \sqrt{t_{j-1}})\snorm[\beta]{v}
 \le \frac{Ch^{\beta}}{\sqrt{t_{j}}} \snorm[\beta]{v}.
\end{split}
\end{equation*}
So
\begin{equation}\label{set2}
k \sum_{j=1}^n t_{j-1}\norm{\eta_j}^2 \le Ch^{2\beta}  t_n \snorm[\beta]{v}^2,\quad
 k \sum_{j=1}^n \norm{\eta_j} \le Ch^{\beta} t_n^{\frac{1}{2}}\snorm[\beta]{v}.
\end{equation}
By using \eqref{bddness} we have, for $j \ge 2$,
\begin{equation*}
\begin{split}
\norm{\delta_j}
 &\le \Norm{\frac{1}{k}\int_{t_{j-1}}^{t_j} (\tau - t_{j-1})Gu_{tt}(\tau)\dd \tau}
 \le \int_{t_{j-1}}^{t_j} \norm{A^{3-\frac{\beta}{2}} E(\tau) A^{\frac{\beta}{2}} v}\, \dd \tau\\
 &\le C \int_{t_{j-1}}^{t_j} \tau^{\frac{-6+\beta}{4}}\,\dd \tau\snorm[\beta]{v},
\end{split}
\end{equation*}
so that, by H\"older's inequality with $p=\frac{4}{\beta}$ and $q = \frac{4}{4-\beta}$, $1 \le \beta <4$,
\begin{align*}
\int_{t_{j-1}}^{t_j}\tau^{\frac{-6+\beta}{4}}\,\dd \tau
  \le Ck^{\frac{\beta}{4}}\Big( \int_{t_{j-1}}^{t_j} 
\big( \tau^{\frac{-6 + \beta}{4}} \big)^{\frac{4}{4-\beta}} \, \dd \tau \Big)^{\frac{4-\beta}{4}}
\le Ck^{\frac{\beta}{4}}\Big(
\frac{\beta-4}{2}\big(t_{j-1}^{-\frac{2}{4-\beta}} 
- t_{j}^{-\frac{2}{4-\beta}}\big) \Big)^{\frac{4-\beta}{4}}
 \le Ck^{\frac{\beta}{4}} t_{j-1}^{-\frac{1}{2}}.
\end{align*}
The same result is obtained with $\beta = 4$. For $j=1$ we have
\begin{align*}
 \norm{\delta_1}
 \le \Norm{\frac{1}{k} \int_0^k \tau G u_{tt}(\tau)\,\dd \tau}
  \le C \frac{1}{k} \int_0^k \tau^{\frac{-2+\beta}{4}}\,\dd \tau \snorm[\beta]{v}
 \le C\frac{4}{2+\beta} k^{\frac{-2+\beta}{4}}\snorm[\beta]{v}
  \le Ck^{\frac{\beta}{4}} t_1^{-\frac{1}{2}}\snorm[\beta]{v}.
\end{align*}
So we have, for $j \ge 1$,
\begin{equation*}
\norm{\delta_j} \le Ck^{\frac{\beta}{4}} t_{j}^{-\frac{1}{2}}\snorm[\beta]{v}.
\end{equation*}
Hence, 
\begin{equation}\label{set3}
 k \sum_{j=1}^n \norm{\delta_j} \le ck^{\frac{\beta}{4}}t_n^{\frac{1}{2}}\snorm[\beta]{v},\quad
 k \sum_{j=1}^n t_{j-1} \norm{\delta_j}^2 \le Ck^{\frac{\beta}{2}}  t_n \snorm[\beta]{v}^2.
\end{equation}
Put \eqref{set1},\, \eqref{num3},\, \eqref{set2}, and \eqref{set3} in \eqref{3.42}, to get
\begin{equation*}
 \norm{e_n} \le C(h^{\beta} +k^{\frac{\beta}{4}}) \snorm[\beta]{v}.
\end{equation*}
This completes the proof \eqref{lemp1}.

To prove \eqref{lemp2} we recall \eqref{3.41} and let $v \in \dot{H}^{\beta-2}$. For the first term we  write  $ k \sum_{j=1}^n \norm{\rho_j}^2 = k \norm{\rho_1}^2 + k \sum_{j=2}^n \norm{\rho_j}^2$, where by \eqref{bddness}
\begin{equation*}
 k \norm{\rho_1}^2
 \le kCh^{2\beta} \norm{AE(k) A^{\frac{\beta-2}{2}}v}^2
 \le Ch^{2\beta} \snorm[\beta-2]{v},
\end{equation*}
and
\begin{align*}
 k \sum_{j=2}^n \norm{\rho_j}^2
 & = \sum_{j=2}^n \int_{t_{j-1}}^{t_j} \Norm{\rho(s) + \int_s^{t_j} \rho_t(\tau) \,\dd \tau}^2\, \dd s\\
 & \le  2\sum_{j=2}^n \int_{t_{j-1}}^{t_j}\norm{\rho(s)}^2\,\dd s + 2\sum_{j=2}^n \int_{t_{j-1}}^{t_j} \Norm{\int_s^{t_j} \rho_t(\tau) \,\dd \tau}^2\, \dd s\\
& \le  2\int_{t_1}^{t_n}\norm{\rho(s)}^2\,\dd s
  + 2\sum_{j=2}^n \int_{t_{j-1}}^{t_j} (t_j -s) \int_{t_{j-1}}^{t_j} \norm{\rho_t(\tau)}^2 \,\dd \tau \, \dd s\\
 & \le 2\int_0^{t_n}\norm{\rho}^2\,\dd\tau + 2k\int_{t_1}^{t_n} \tau \norm{\rho_t}^2 \,\dd \tau,
\end{align*}
since $t_j -s \le k \le \tau$ and  where, by \eqref{2.18},
\begin{equation*}
\begin{aligned}
& \int_{0}^{t_n}\norm{\rho}^2\,\dd \tau
\le Ch^{2\beta} \snorm[\beta-2]{v}^2,\\
\end{aligned}
\end{equation*}
and
\begin{equation*}
\begin{aligned}
 k\int_{t_1}^{t_n} \tau \norm{\rho_t}^2 \,\dd \tau
 &\le Ch^{2\beta}k \int_{t_1}^{t_n} \tau \norm{A^3 E(\tau) A^{\frac{\beta-2}{2}}v}^2 \,\dd \tau\\
& \le Ch^{2\beta} k \int_{k}^{t_n} \tau^{-2} \,\dd \tau\, \snorm[\beta-2]{v}^2
 \le Ch^{2\beta} k (k^{-1} - t_n^{-1})\snorm[\beta-2]{v}^2 
 \le Ch^{2\beta} \snorm[\beta-2]{v}^2.
\end{aligned}
\end{equation*}
So
\begin{equation}\label{3.49}
  k\sum_{j=1}^n \norm{\rho_j}^2 \le Ch^{2\beta} \snorm[\beta-2]{v}^2 .
\end{equation}
Now we compute $k\sum_{j=1}^n \norm{\eta_j}$. Recall that $\eta_j = -(R_h-I)G\pt u_j$ and $\eta = -(R_h - I)Gu_t$, so
\begin{align*}
\norm{\eta_j}
  = \Norm{(R_h-I)G\frac{1}{k}\int_{t_{j-1}}^{t_j} u_t\,\dd \tau}
   \le \frac{1}{k}\int_{t_{j-1}}^{t_j} \norm{(R_h-I)Gu_t}\,\dd \tau
  \le \frac{1}{k} \int_{t_{j-1}}^{t_j} \norm{\eta}\,\dd \tau,
\end{align*}
and hence by  \eqref{yek} we have
\begin{equation}\label{3.50}
 k \sum_{j=1}^n \norm{\eta_j} \le \int_0^{t_n}\norm{\eta}\,\dd\tau \le Ch^{\beta} \abs{\log h} \snorm[\beta-2]{v}.
\end{equation}
For computing $k\sum_{j=1}^n\norm{\delta_j}$  we use \eqref{bddness} and obtain for $1 \le \beta <4$,
\begin{align*}
 \norm{\delta_j}
 \le \frac{1}{k} \int_{t_{j-1}}^{t_j} (\tau - t_{j-1}) \norm{Gu_{tt}(\tau)}\,\dd \tau
  \le \int_{t_{j-1}}^{t_j} \norm{A^{4-\frac{\beta}{2}} E(\tau) A^{\frac{\beta-2}{2}}v}\,\dd \tau
 \le C \int_{t_{j-1}}^{t_j} \tau^{-2 +\frac{\beta}{4}}\,\dd \tau\, \snorm[\beta-2]{v}.
 \end{align*}
 Hence,
 \begin{equation*}
k \sum_{j=2}^n \norm{\delta_j}
  \le  Ck \int_{k}^{t_n} \tau^{-2+\frac{\beta}{4}} \, \dd \tau\snorm[\beta-2]{v}
  \le  Ck\frac{4}{4-\beta} \Big( k^{-1+\frac{\beta}{4}} -
  t_n^{-1+\frac{\beta}{4}} \Big)\snorm[\beta-2]{v} 
 \le  \frac{C}{4-\beta}k^{\frac{\beta}{4}}\snorm[\beta-2]{v} 
\end{equation*}
and
\begin{equation*}
k\norm{\delta_1}
 \le \int_0^k \tau\norm{Gu_{tt}(\tau)}\,\dd \tau
 \le \int_0^k \tau \norm{A^{4-\frac{\beta}{2}} E(\tau) A^{\frac{\beta-2}{2}}v}\,\dd \tau
 \le C\int_0^k \tau^{\frac{\beta}{4} -1}\,\dd \tau\, \snorm[\beta-2]{v}
 \le \frac{C}{4-\beta}k^{\frac{\beta}{4}}\snorm[\beta-2]{v}.
\end{equation*}
Therefore, for $1 \le \beta < 4$,
\begin{equation*}
 k \sum_{j=1}^n \norm{\delta_j} \le \frac{C}{4 - \beta} k^{\frac{\beta}{4}} \snorm[\beta -2]{v}.
\end{equation*}
If we put $\frac{1}{4-\beta} = \abs{\log k}$, we also have
\begin{align*}
 k \sum_{j=1}^n \norm{\delta_j}
 \le  \frac{C}{4-\beta} k^{1-\frac{4-\beta}{4}} \snorm[\beta-2]{v}
 =  C\abs{\log k} k \ee^{-\frac{4-\beta}{4}\log k}\snorm[\beta-2]{v}
 \le Ck\abs{\log k}\snorm[\beta -2]{v}
= C \abs{\log k} \snorm[\beta-2]{v}.
\end{align*}
Therefore, for $1\le \beta \le 4$, we have
\begin{equation}\label{3.51}
k \sum_{j=1}^n \norm{\delta_j} \le C_{\beta,k}k^{\frac{\beta}{4}}\snorm[\beta-2]{v}.
\end{equation}
where $C_{\beta,k} = \frac{C}{4-\beta}$ for $\beta < 4$ and $C_{\beta,k} = C\abs{\log k}$ for $\beta =4$.
Finally we put \eqref{3.49}, \,\eqref{3.50} and \eqref{3.51} in \eqref{3.41}, to get
\begin{equation*}
\Big( k \sum_{j=1}^n \norm{e_j}^2  \Big)^{\frac{1}{2}}
\le \Big( Ch^{\beta} \abs{\log h} + C_{\beta,k}k^{\frac{\beta}{4}}\Big)\snorm[\beta-2]{v}.
\end{equation*}
This completes the proof.  
\end{proof}
\section{Approximation of the linear Cahn-Hilliard-Cook
  equation}\label{sec:3}
Consider the linear Cahn-Hilliard-Cook equation \eqref{1.4} with mild solution
\begin{equation}\label{4.1}
  X(t) = E(t)X_0 +  \int_0^t E(t-s)\,\dd W(s).
\end{equation}
We recall the isometry of the It\^o integral,
\begin{equation}\label{isometry}
\IE\Big\{\Norm{\int_0^t B(s) \,\dd W(s)}^2 \Big\}
= \IE \Big\{\int_0^t \norm[\HS]{B(s) Q^{\frac{1}{2}}}^2\,\dd s\Big\},
\end{equation}
where the Hilbert-Schmidt norm is defined by
\begin{equation} \label{HSdef}
\norm[\HS]{T}^2 = \sum_{l=1}^{\infty} \norm{T \phi_l}^2.
\end{equation}
Here $\lbrace\phi_l\rbrace_{l=1}^{\infty}$ is an arbitrary  orthonormal basis for $H$.
In the next theorem we consider the regularity of the mild solution
\eqref{4.1}.  The $L_2(\Omega ,\dot{H}^\beta)$-norm is defined in \eqref{meansquare}.
\begin{theorem}\label{thm1}
 Let $X(t)$ be the mild solution \eqref{4.1} with $X_0 \in L_2(\Omega ,
 \dot{H}^{\beta})$ and $\norm[\HS]{A^{\frac{\beta - 2}{2}}
   Q^{\frac{1}{2}}} < \infty$ for some $\beta \ge 0$.  Then
 \begin{equation*}
 \norm[L_2(\Omega , \dot{H}^{\beta})]{X(t)}
\le C\Big( \norm[L_2(\Omega , \dot{H}^{\beta})]{X_0} + \norm[\HS]{A^{\frac{\beta - 2}{2}} Q^{\frac{1}{2}}} \Big), \quad t \ge 0.
\end{equation*}
Moreover, if $\beta=0$, then for the norm in $H$ we have 
 \begin{equation*}
 \norm[L_2(\Omega ,H)]{X(t)}
\le C\Big( \norm[L_2(\Omega , H)]{X_0} 
+ 
\norm[\HS]{A^{-1} Q^{\frac{1}{2}}}+t^{\frac12} \Big), \quad t \ge 0.
\end{equation*}
\end{theorem}
\begin{proof}  Recall the definition of $\snorm[\beta]{\cdot}$ in \eqref{defnorm}.  
By using the isometry \eqref{isometry}, the definition of the
Hilbert-Schmidt norm \eqref{HSdef}, and \eqref{bddness},\, \eqref{half} we get, for
$\beta\ge0$, see \eqref{fix}, 
\begin{align*}
  \norm[L_2(\Omega , \dot{H}^{\beta})]{X(t)}^2
  & = \IE \Big\{\Big|E(t) X_0 + \int_0^t E(t-s) \,\dd W(s)\Big|_{\beta}^2\Big\}\\
  & \le C \Big(\IE \big\{\big\|A^{\frac\beta2}PE(t) X_0\big\|^2\big\} + \IE \Big\{ \Norm{\int_0^t A^{\frac{\beta}{2}}P E(t-s) \,\dd W(s)}^2\Big\} \Big)\\
  & \le  C \Big(\norm[L_2(\Omega , \dot{H}^{\beta})]{X_0}^2 +  \int_0^t\norm[\HS]{A^{\frac{\beta}{2}} E(s)P Q^{\frac{1}{2}}}^2 \,\dd s \Big)\\
  & \le  C \Big(\norm[L_2(\Omega , \dot{H}^{\beta})]{X_0}^2 +  \sum_{l=1}^{\infty}\int_0^t\norm{A^{\frac{\beta}{2}} E(s)P Q^{\frac{1}{2}}\phi_l}^2 \,\dd s \Big)\\
 & \le  C \Big(\norm[L_2(\Omega , \dot{H}^{\beta})]{X_0}^2 +  \sum_{l=1}^{\infty} \norm{A^{\frac{\beta - 2}{2}} Q^{\frac{1}{2}} \phi_l}^2 \Big)\\
  & =  C \Big(\norm[L_2(\Omega , \dot{H}^{\beta})]{X_0}^2 +  \norm[\HS]{A^{\frac{\beta - 2}{2}} Q^{\frac{1}{2}}}^2 \Big).
\end{align*}
For $\beta=0$ and the $H$-norm, there are  additional terms
\begin{align*}
\IE\{\norm{(I-P)X_0}^2\}
&= \IE\{(X_0,\varphi_0)^2\}
\le \norm[L_2(\Omega , H)]{X_0}^2 ,\\
\IE\{\norm{(I-P)W(t)}^2\}
&= \IE\{(W(t),\varphi_0)^2\}
\le  Ct.
\end{align*}
The proof is complete.  
\end{proof}

The finite element approximation of the linear Cahn-Hilliard-Cook
equation is: Find $X_h(t) \in S_h$ such that
\begin{equation}\label{FMCH}
\dd X_h + A_h^2 X_h \, \dd t = P_h \,\dd  W,\quad t>0;\quad 
X_h(0) = P_h X_0, 
\end{equation}
with the mild solution 
\begin{equation}\label{4.2}
 X_h(t) = E_h(t) P_h X_0 + \int_0^t E_h(t-s) P_h \,\dd W(s).
\end{equation}
Recall that $r\ge2$ is the order of the finite element method,
cf.~\eqref{ritzerr}.

\begin{theorem}\label{thm2}
Let $X_h$ and $X$ be the mild solutions \eqref{4.2} and \eqref{4.1}
with $X_0 \in L_2(\Omega , \dot{H}^{\beta})$ and assume that 
$\norm[\HS]{A^{\frac{\beta - 2}{2}} Q^{\frac{1}{2}}} < \infty$  for
some $\beta \in[1,r]$.
Then there are $h_0$ and $C$, such
that, for $h\le h_0$ and $t \ge 0$, 
\begin{equation*}
 \|X_h(t) - X(t)\|_{L_2(\Omega , H)}
 \le Ch^{\beta}\big( \norm[L_2(\Omega , \dot{H}^{\beta})]{X_0} 
 + \abs{\log h}\norm[\HS]{A^{\frac{\beta - 2}{2}} Q^{\frac{1}{2}}} \big).
\end{equation*}
\end{theorem}
\begin{proof}
 Use \eqref{4.1} and \eqref{4.2} and set $F_h(t) = E_h(t) P_h - E(t)$ to get
 \begin{equation*}
 \norm[L_2(\Omega , H)]{X_h(t) - X(t)} \le \norm[L_2(\Omega , H)]{e_1(t)} + \norm[L_2(\Omega , H)]{e_2(t)},
\end{equation*}
where, see \eqref{errorX},
\begin{align*}
e_1(t) &= F_h(t) X_0 = F_h(t) PX_0,\\
e_2(t) &= \int_0^tF_h(t-s) \,\dd W(s) = \int_0^t F_h(t-s)P \,\dd W(s).
\end{align*}
By using Theorem
\ref{lem1} we get
\begin{equation*}\label{4.4}
\norm[L_2(\Omega , H)]{e_1(t)}
 = \big( \IE \norm{F_h(t) X_0}^2 \big)^{\frac{1}{2}}
\le Ch^{\beta} \big( \IE \snorm[\beta]{X_0}^2 \big)^{\frac{1}{2}}
 = Ch^{\beta} \norm[L_2(\Omega , \dot{H}^{\beta})]{X_0}.
\end{equation*}
For the second term we use the isometry \eqref{isometry}, the
definition of the Hilbert-Schmidt norm \eqref{HSdef}, and Theorem
\ref{lem1},
\begin{align*}
 \norm[L_2(\Omega , H)]{e_2(t)}^2
 & = \IE \Big( \Norm{\int_0^t F_h(t-s) \,\dd W(s)}^2\Big) \\
 & = \int_0^t \norm[\HS]{F_h(t-s) Q^{\frac{1}{2}}}^2 \,\dd s
 = \sum_{l=1}^{\infty} \int_0^t \norm{F_h(s) Q^{\frac{1}{2}} \phi_l}^2 \,\dd s\\
 & \le C\abs{\log h}^2h^{2\beta} \sum_{l=1}^{\infty} 
   \snorm[\beta - 2]{Q^{\frac{1}{2}}\phi_l}^2
  = C\abs{\log h}^2h^{2\beta} \norm[\HS]{A^{(\beta - 2)/2}Q^{\frac{1}{2}}}^2.
\end{align*}
The proof is complete. 
\end{proof}

Now we consider the fully discrete Cahn-Hilliard-Cook equation \eqref{DeltaXhn} with mild solution
\begin{equation}\label{4.6}
  X_{h,n} = E_{kh}^n P_h X_0 + \sum_{j=1}^n E_{kh}^{n-j+1} P_h
  \,\delta W_j,
\quad \text{where $E_{kh} = (I + kA_h^2)^{-1}$.}
\end{equation}

\begin{theorem}\label{thm3}
  Let $X_{h,n}$ and $X$ be given by \eqref{4.6} and \eqref{4.1} with
  $X_0 \in L_2(\Omega , \dot{H}^{\beta})$ and
  $\norm[\HS]{A^{\frac{\beta - 2}{2}}Q^{\frac{1}{2}}} < \infty$ for
  some $\beta \in [1, \min(r,4)]$.  Then there are $h_0,k_0$ and $C$, such
that, for $h\le h_0$, $k\le k_0$, and $n \ge 1$, 
\begin{align*}
  \|X_{h,n} -  X(t_n)  \|_{L_2(\Omega , H)}
    \le \big(C\abs{\log h} h^{\beta} 
 + C_{\beta,k}k^{\frac{\beta}{4}}\big)
\big( \norm[L_2(\Omega , \dot{H}^{\beta})]{X_0} 
+ \norm[\HS]{A^{\frac{\beta - 2}{2}}Q^{\frac{1}{2}}} \big),
\end{align*}
where $C_{\beta,k} = \frac{C}{4-\beta}$ for $\beta <4$ and $C_{\beta,k} = C\abs{\log k}$ for  $\beta =4$.
\end{theorem}
\begin{proof}
By using \eqref{4.1} and \eqref{4.6} we get, with $F_n = E_{kh}^n P_h - E(t_n)$,
\begin{align*}
 e_n
  = F_n X_0 +  \sum_{j=1}^n  \int_{t_{j-1}}^{t_j}  F_{n-j+1} \, \dd W(s)
  + \sum_{j=1}^{n}\int_{t_{j-1}}^{t_j} \big( E(t_n - t_{j-1}) -
 E(t_n -s)\big) \, \dd W(s)
= e_{n,1} + e_{n,2} + e_{n,3}.
 \end{align*}
By using Theorem \ref{lemma2} we have
\begin{equation}\label{4.8}
  \norm[L_2(\Omega , H)]{e_{n,1}} = \big( \IE \norm{F_n X_0}^2 \big)^{\frac{1}{2}}
\le C(h^{\beta}  + k^{\frac{\beta}{4}})\norm[L_2(\Omega , \dot{H}^{\beta})]{X_0}.
\end{equation}
By using the isometry \eqref{isometry} and Theorem \ref{lemma2} we get
\begin{align*}
 \norm[L_2(\Omega , H)]{e_{n,2}}^2
  &= \IE \Big( \Norm{ \sum_{j=1}^n  \int_{t_{j-1}}^{t_j} F_{n-j+1} \, \dd W(s) }^2\Big)\\
 &= \sum_{j=1}^n \int_{t_{j-1}}^{t_j} \norm[\HS]{F_{n-j+1} Q^{\frac{1}{2}}}^2 \,\dd s
  = k \sum_{l=1}^{\infty} \sum_{j=1}^n  \norm{F_{n-j+1} Q^{\frac{1}{2}} \phi_l}^2\\
 & \le \sum_{l=1}^{\infty} \big(C\abs{\log h}h^{\beta} 
 + C_{\beta,k}k^{\frac{\beta}{4}}\big)^2 \snorm[\beta - 2]{ Q^{\frac{1}{2}} \phi_l}^2\\
&  =  \big(C\abs{\log h}h^{\beta} 
  + C_{\beta,k}k^{\frac{\beta}{4}}\big)^2 \norm[\HS]{A^{\frac{\beta - 2}{2}} Q^{\frac{1}{2}}}^2.
 \end{align*}
By using the isometry property \eqref{isometry} again we have
\begin{equation*}
\begin{split}
 \|e_{n,3}  \|_{L_2(\Omega , H)}^2
 &\le  \IE \Big( \Norm{\sum_{j=1}^n \int_{t_{j-1}}^{t_j} 
  ( E(t_n - t_{j-1}) - E(t_n -s)) \, \dd W(s)}^2\Big)\\
 & = \sum_{j=1}^n \int_{t_{j-1}}^{t_j} 
  \norm[\HS]{(E(t_n - t_{j-1}) - E(t_n -s)) Q^{\frac{1}{2}}}^2  \, \dd s\\
 & = \sum_{l=1}^{\infty} \sum_{j=1}^n  \int_{t_{j-1}}^{t_j} \norm{A^{-\frac{\beta}{2}}(E(s - t_{j-1}) - I) A E(t_n -s) A^{\frac{\beta -2}{2}} Q^{\frac{1}{2}} \phi_l}^2  \, \dd s.
\end{split}
\end{equation*}
By using the well-known inequality
\begin{equation*}
 \norm{A^{\frac{-\beta}{2}} \big( E(t) -I\big) w} \le Ct^{\frac{\beta}{4}}\norm{w},
\end{equation*}
 with  $t = s - t_j, w =  A E(t_n -s) A^{\frac{\beta -2}{2}} Q^{\frac{1}{2}} \phi_l$, together with \eqref{half}, we get
\begin{align*}
 \norm[L_2(\Omega , H)]{e_{n,3}}^2
 & \le C k ^{\frac{\beta}{2}} \sum_{l=1}^{\infty} \int_0^{t_n} \norm{A E(t_n -s) A^{\frac{\beta -2}{2}} Q^{\frac{1}{2}} \phi_l}^2 \, \dd s\\
 & \le C k ^{\frac{\beta}{2}} \sum_{l=1}^{\infty} \norm{A^{\frac{\beta -2}{2}} Q^{\frac{1}{2}} \phi_l}^2
 = C k ^{\frac{\beta}{2}} \norm[\HS]{A^{\frac{\beta -2}{2}} Q^{\frac{1}{2}}}^2.
 \end{align*}
Putting these together proves the desired result.
\end{proof}

\section{Conclusions} \label{sec:4} 
We have studied numerical approximation of the linearized
Cahn-Hilliard-Cook equation by a spatially semidiscrete finite element
method and a completely discrete method based on the implicit Euler
time-stepping.  We have proved strong convergence estimates of optimal
order except for a logarithmic factor.  

By means of the It\^o isometry the proofs are reduced to proving error
estimates for the corresponding deterministic problem, that is, error
estimates for approximations of the Cahn-Hilliard semigroup $\ee^{-tA^2}$.  
This is where the main effort has been spent.  Since the finite
element method is based on $(A_h)^2$ rather than $(A^2)_h$ this analysis
is more difficult than for the linear heat equation in \citet{yubin2}.  

Our results should be viewed as results on approximation of the
stochastic convolution \eqref{stochconvol}, which is a part of the
mild solution of the nonlinear Cahn-Hilliard-Cook equation.  The
remaining part, which solves a nonlinear random evolution problem, is
studied in \citet{KLMchc}, where strong convergence is proved for the
spatially semidiscrete approximation of the nonlinear
Cahn-Hilliard-Cook equation, but without known rate of convergence.
To obtain the optimal rate of convergence remains a challenge.
Another open problem is to study weak convergence.


\end{document}